\documentclass[11pt]{amsart}
\usepackage{color}

\usepackage[left=3.5cm,top=3.5cm,right=3.5cm,bottom=3.5cm]{geometry}

\usepackage[utf8]{inputenc}

\newtheorem{thm}{Theorem}
\newtheorem{col}{Corollary}
\newtheorem{lem}{Lemma}
\newtheorem{pr}{Proposition}
\theoremstyle{remark}
\newtheorem{rem}{Remark}
\newtheorem{example}{Example}

\def\R{\mathbb{R}}

\def\D{\mathcal{D}}
\def\N{\mathbb{N}}
\def\E{\mathbb{E}}
\def\norm#1{\|#1\|}
\def\ilsk#1#2{\left<#1 , #2\right>}

\def\1{1{\hskip -2.5 pt}\hbox{\textup{l}}}

\def\sgn{\operatorname{sgn}}

\begin{document}

\title{Model of phenotypic evolution in hermaphroditic populations}
\author[R. Rudnicki]{Ryszard Rudnicki}
\author[P. Zwole\'nski]{Pawe{\l} Zwole{\'n}ski}
\address{Institute of Mathematics,
Polish Academy of Sciences,
Bankowa 14, 40-007 Katowice, Poland.}
\email{rudnicki@us.edu.pl}
\email{pawel.zwolenski@gmail.com}
\thanks{This research was partially supported
by the State Committee for
Scientific Research (Poland) Grant No.~N~N201~608240 (RR)}
\keywords{measure valued process, phenotypic evolution, sexual population model, nonlinear transport equation,
asymptotic stability}
\subjclass[2010]
{Primary 47J35:  Secondary: 34G20, 60K35, 92D15}  
\begin{abstract}
We consider an individual based model of phenotypic evolution in herma\-phro\-ditic populations
which includes random and assortative 
mating of individuals. 
By increasing the number of individuals to infinity
we obtain a nonlinear transport equation, 
which describes the evolution of distribution of phenotypic traits. 
Existence of an one-dimensional attractor is proved 
and the formula for the density of phenotypic traits in the limiting (asymptotic) population 
is derived in some particular case.
\end{abstract}
\maketitle

\section{Introduction}
In this paper we study an evolution of the phenotypic traits in the hermaphroditic populations, 
i.e. the populations in which every individual has both male and female reproductive system. 
A great part of that kind of populations are characterized by the formation 
of various defense mechanisms against self-fertilization (\textit{autogamy}) 
to guarantee the genetic diversification 
(e.g. the proper shape of the flower can inhibit a self-pollination in some species of plants). 
In that case individuals can only mate with others to copulate and cross-fertilize. 
Nonetheless, some of the hermaphroditic populations have an ability 
of both mating/cross-fertilization and the self-fertilization. 
In the present paper we are interested in these both types of populations: with or without self-fertilization. 
The hermaphroditic populations are plentiful among both water or terrestrial animals and plants, 
and just to set an example we can mention Sponge (\textit{Porifera}), \textit{Turbelleria}, 
Cestoda (\textit{Cestoidea}), \textit{Lumbricidae}, some of the mollusks such as sea slug Blue Dragon (\textit{Glaucus atlanticus}) and various kinds of land snails or majority of flowering plants (angiosperms). 
 
Lots of individual based models for asexual populations were studied as the microscopic probabilistic 
descriptions of the evolution of the individuals' traits 
and the macroscopic approximations of them were derived in the forms of deterministic processes or superprocesses 
(see \cite{champagnat2,champagnat,fournier_meleard,ferriere_tran,meleard_tran}). 
In this paper we formulate an individual based model to describe the phenotypic evolution in the hermaphroditic populations. 
Precisely, we consider a large population of small individuals characterized by their phenotypic traits, 
which are assumed to be independent of spacial location and unchanged during the lifetime. 
The skin color, the shape of a leaf and the pattern of a shell constitute examples 
of the phenotypic traits that might be of interest of us. All the individuals are capable of mating with others 
or self-fertilizing to give birth to an offspring. 
We consider a general model of mating which includes both random and assortative mating.
The first particular case is a semi-random mating model which is based on the assumption that each individual has an initial capability $p(x)$ of mating depending  on 
its phenotypic trait $x$. This mating model is similar to models describing aggregation processes in phytoplankton dynamics
(see \cite{arino_rudnicki,rudnicki_wieczorek,rudnicki_wieczorek2}).
The second  particular case is an assortative mating model, when individuals with similar
phenotypic traits mate more often than they would choose a partner randomly. 
We adapt to  hermaphroditic populations  the model based on  a preference function 
 \cite{DBLD,GB,MGG,PB,SB,SP},
which is usually used in two-sex populations. 
The consequence of mating or self-fertilization is a birth of a new individual with a phenotypic trait
given by a random variable, which depends only on the phenotypic traits of the parents.
Moreover, each individual can die naturally or due to a competition with other individuals. 
We consider a continuous time model and assume that all above-mentioned events happen randomly. 

The model that we study here is a hermaphroditic analogue of the asexual model introduced 
by Bolker and Pacala \cite{bolker_pacala}, Dieckmann and Law \cite{dieckmann_law} 
and studied by Fournier and 
M\'el\'eard \cite{fournier_meleard}. 
Although the literature connected with the individual based models and the macroscopic approximations is vast, 
at our knowledge no similar model, involving the mating process, has been considered yet. 

The another aim of the paper is to study macroscopic deterministic approximation of the model, 
obtained by increasing the number of the individuals in population to infinity, with simultaneous decrease
in the mass of each single individual.
After suitable scaling of the parameters of the model the limit passage leads 
to an integro-differential equation, 
whose solutions describe the evolution of the distribution of the phenotypic traits. 
Furthermore, not only do we study existence, uniqueness of the solutions 
but also investigate 
 extinction and  persistence of the population and  convergence of its size  to some stable level.

The main aim of our paper is to prove the asymptotic stability of the distribution of phenotypic traits. 
It turns out that the asymptotic behavior of the solutions 
is characterized by the conservation of the first moment of the phenotypic distribution. 
We apply our general result to two specific models
when the phenotypic trait 
of the offspring is the mean of the phenotypic traits of parents randomly perturbed  by some external environmental effects or genetic mutations.
The noise in the first model is  additive. In this case we are also able to derive the formula for the phenotypic distribution of the limiting population.
The second model corresponds to multiplicative noise. 
A similar model describes the density distribution function
of the particles' energy, and it is known as the Tjon--Wu version 
of classical Boltzmann equation  (see \cite{boby, krook, tjon}). 
As a by-product of our investigation we give a simple proof of the theorem of
Lasota and Traple (see \cite{lasota, lastrap}) concerning 
  asymptotic stability of this equation.  
 We also give an example of the situation when in a long period of time all phenotypic traits
in the population reduce to one particular phenotypic trait, which is the mean trait of the initial population. 

The scheme of the paper is following. In Section \ref{r_def} we collect the assumptions concerning the dynamics of the population.
In Section \ref{r_mod} we introduce a stochastic process corresponding to our individual based model and study  its properties. 
Section \ref{r_app} is devoted to the macroscopic  approximation,  the limiting equation and its solutions.
In particular, we  give some simple results concerning 
extinction and persistence of the population and stabilization of its size.
Section~\ref{s:as} contains some results concerning the asymptotic stability
and examples of applications.
In the last section we discuss some  problems for future investigation  concerning  
assortative mating models.

\section{Individual-based model}
\label{r_def}
Let us fix a positive integer $d$. We assume that every individual is described by 
a phenotypic trait $x$ which belongs to some closed and connected subset $F$ of $\R^d$, 
whose interior is nonempty. The phenotypic trait of an individual does not depend
on the spatial location and does not change during a lifetime.

\subsection{Random mating.} 
\label{ss:rm}
In sexually reproducing populations the mating process highly depends on a given species. 
We will consider both random and assortative mating.  
In classical genetics individuals mate randomly --- the choice of partner is not influenced by the phenotypic traits
(\textit{panmixia}). Random mating better describes plants than animals, but it is also observed 
in some animal hermaphroditic populations  \cite{Baur}.  
We study a semi-random mating model in which the mating rate depends on the phenotypic trait.
An individual described by the phenotypic trait $x$ is capable 
of mating/self-fertilizing with the rate $p(x)$, where $p$ is a positive function 
of the phenotypic trait.

Consider a population which consists of  $n$  individuals with phenotypic traits $x_1,\dots,x_n$. 
Since two different individuals can have the same 
phenotypic trait, it is useful to describe the  state of the population as the
multiset
\[
\mathbf x=\{x_1,\dots,x_n\}.
\]     
We recall that a \textit{multiset} 
(or \textit{bag}) is a generalization of the notion of set in which members are allowed to appear more than once.
We suppose that an individual can mate with individual with phenotypic trait $x_j$ with the following probability 
$$\frac{p(x_j)}{\sum_{l=1}^n p(x_l)}.$$
Thus the mating rate of the individuals characterized by phenotypic traits $x_i$ and $x_j$ is given by
\begin{equation}
\label{matingop1}
m(x_i,x_j;\mathbf x)=\frac{p(x_i)p(x_j)}{\sum_{l=1}^n p(x_l)}.
\end{equation}
The figure $m(x_i,x_i;\mathbf x)$ is the self-fertilization rate. In the case of populations without self-fertilization we posit that
\begin{equation}
\label{matingop2}
m(x_i,x_j;\mathbf x)=\frac{p(x_i)p(x_j)}{2}\bigg(\frac{1}{\sum_{l\neq i} p(x_l)}+\frac{1}{\sum_{l\neq j} p(x_l)}\bigg)
\end{equation}
if $i\ne j$ and $m(x_i,x_i;\mathbf x)=0$. 
Let us observe that in both cases the mating rate is a symmetric function of $x_i$ and $x_j$ but only
in the first case we have  $\sum_{j=1}^{n}m(x_i,x_j;\mathbf x)=p(x_i)$.
If we  pass with the number of individuals to infinity and replace the discrete model by the infinitesimal model 
 with the phenotypic trait distribution described by a continuous measure $\mu$, then
 the mating rate in both cases is given by
\begin{equation}
\label{mating-rd}
m(x,y;\mu)=\frac{p(x)p(y)}{\int_F p(z)\mu(dz)}.
\end{equation}
 

\subsection{Assortative mating.} 
\label{ss:am}
Now we consider models with assortative mating,  i.e.  when individuals with similar phenotypic traits mate more often than they would choose a partner randomly.
Assortative  mating can be modeled in different ways. For example 
one can use the  matching theory, which is based on preference lists, whereby each participant
ranks all potential partners according to its
preferences and attempts to pair with the highest-ranking
partner \cite{AA,PBG}. 
Such models are very interesting but difficult to analyze.
The most popular  models of assortative mating are based on the assumption that 
 a random encounter between two individuals with phenotypic traits $x$ and $y$ depends
on a preference function  $a(x,y)$ \cite{DBLD,GB,MGG,PB,SB,SP}. We consider only the case when all the individuals have the same initial capability of mating $p(x)=1$.

Usually, it is assumed  that $a(x,y)=\varphi(\|x-y\| )$, where 
$\varphi: [0,\infty)\to [0,\infty)$ is a continuous  decreasing function.
It means that if a population  consists of  $n$  individuals with phenotypic traits $x_1,\dots,x_n$, 
then an individual with phenotypic trait $x_i$ mates with an individual with phenotypic trait  $x_j$ with the rate 
\begin{equation}
\label{mating-as1}
m(x_i,x_j;\mathbf x)=\frac{a(x_i,x_j)}{\sum_{l=1}^n a(x_i,x_l)}=\frac{\varphi(\|x_i-x_j\|)}{\sum_{l=1}^n \varphi(\|x_i-x_l\|)}.
\end{equation}
Note that in general, the function $m$ is not symmetric in $x_i$ and $x_j$, 
and usually it describes mating in a two-sex population. Then
the first argument in $m$ refers to the female.
Females are assumed to mate only once, whereas males may
participate in multiple matings. We have $\sum_{j=1}^n m(x_i,x_j;\mathbf x)=1$ for each $i$,  which  means that
the probability that a female succeeds in mating equals to one.
The mating rate in the infinitesimal model  is of the form
 \begin{equation}
\label{mating-asd}
m(x,y;\mu)=\frac{a(x,y)}{\int_F a(x,z)\mu(dz)}.
\end{equation}

While considering hermaphroditic populations, one can expect a model with a symmetric mating rate.
One possible solution is to assume that the mating rate is of the form
\begin{equation}
\label{mating-as2}
m(x_i,x_j;\mathbf x)=\frac{a(x_i,x_j)}{2\sum_{l=1}^n a(x_i,x_l)}+\frac{a(x_i,x_j)}{2\sum_{l=1}^n a(x_j,x_l)},
\end{equation}
where $a(x,y)$ is a symmetric nonnegative preference function, e.g.  $a(x,y)=\varphi(\|x-y\|)$
(in  the case of populations without self-fertilization we eliminate  from  the denominators terms with $i=l$ and $j=l$).
The mating rate in the infinitesimal model is now of the form
 \begin{equation}
\label{mating-asd2}
m(x,y;\mu)=\frac{a(x,y)}{2\int_F a(x,z)\mu(dz)}+\frac{a(x,y)}{2\int_F a(y,z)\mu(dz)}.
\end{equation}

\medskip

In the rest of the paper we will assume that the mating rate $m(x_i, x_j; \textbf x),$ of the individuals with phenotypic traits $x_i$ and $x_j$, is of the form (\ref{matingop1}) or (\ref{mating-as2}).

\subsection{Birth of a new individual.} 
\label{ss:birth}
After the mating/self-fertilization an offspring is born with the rate $1$. 
The phenotypic trait of the offspring is drawn from a distribution $K(x_i,x_j,dz)$,
 where $x_i$ and $x_j$ are the traits of the parents.
We suppose that for every $x,y\in F$ the measure $K(x,y,\cdot)$ is a Borel probability measure with support contained in the set $F$,
and assume that there exist positive constants $c_1,c_2,c_3$
such that
\begin{equation}
\label{srednia1}
\int_F |z| K(x,y,dz)\le c_1+c_2|x|+c_3|y|,
\end{equation}
and
\begin{equation}
\label{srednia2}
\int_F z K(x,y,dz)=\frac{x+y}{2}.
\end{equation}
The above condition has simple biological interpretation, 
namely, the mean value  of an offspring's phenotypic trait is the mean 
 of phenotypic traits of its parents. Moreover, we suppose that for every $x,y\in F$ and for every Borel subset $A$ of the set $F$
\begin{equation}
K(x,y,A)=K(y,x,A),
\end{equation}
and the function
\begin{equation}\label{Kmierzalne}
(x,y)\mapsto K(x,y,A)
\end{equation} is measurable.

\subsection{Competition and death rates.}
\label{ss:cd}
Any individual from the population can die naturally or due to loosing an intra-specific competition with other individuals.
 Let us denote by $I(x_i)$ the rate of interaction of the individual with phenotypic trait $x_i$. 
We assume that $I$ is a nonnegative function of phenotypic trait. For the individuals with phenotypic traits $x_i$ and $x_j$ we define a competition kernel $U(x_i,x_j),$
which is assumed to be a nonnegative and symmetric function.
The competition of any two individuals always ends with death of one of the competitors.

We assume that the natural death rate of the individual with phenotypic trait $x_i$ is expressed by 
the number $D(x_i)$, and suppose that $D$ is nonnegative function of phenotypic trait.

\section{Stochastic process corresponding to the model}
\label{r_mod} 
\subsection{The dynamics of the population}
\label{ss:dp}
Now we present the dynamics of the ecological system that we are interested in. 
The process starts at time $t=0$ from an initial distribution. Individuals with phenotypic traits $x_i$ and $x_j$ 
can mate at the random time, which is exponentially distributed with 
the parameter $m(x_i,x_j;\mathbf x)$, which is of the form (\ref{matingop1})  or (\ref{mating-as2}),
and almost surely after the mating an offspring is born.
The death can occur in two independent ways: the random time of natural death 
of the individual with phenotypic trait $x_i$ is exponentially distributed with the parameter $D(x_i)$; 
the time of competition-caused death is drawn from exponential distribution 
with the parameter $I(x_i)\sum_{j}U(x_i,X_j(t)),$ where the sum extends over all living individuals at time $t$ and $X_j(t)$ 
denotes the phenotypic traits of those individuals at that moment. We assume that all above-mentioned random times 
(mating, natural death, competition death) are mutually independent.

\subsection{The phase space}
\label{ss:p-s}
 By $\N$ we denote the set of all positive integers, 
$\delta_x$ stands for a Dirac measure concentrated at the point $x,$ and $\1_A$ 
denotes a characteristic function of a measurable set $A$. 
We consider a set $M(F)$ of all finite positive Borel measures on the set $F$ equipped with the topology of weak convergence of measures,
 and introduce the subfamily $\mathcal{M}$ of $M(F)$ of the form
\begin{equation}
\mathcal{M}=\bigg\{\sum_{i=1}^n \delta_{x_i}\colon \, n\in\N,\, x_i\in F\bigg\}.
\end{equation}
For any measure $\mu\in\mathcal{M}$ and any measurable function $f$ we define 
$\ilsk{\mu}{f}=\int_F f\, d\mu $, i.e.   $\ilsk{\mu}{f} =\sum_{i=1}^n  f(x_i)$   if  $\mu= \sum_{i=1}^n \delta_{x_i}$. We write $\D([0,\infty), \mathcal{M})$ for the Skorokhod space of all cad-lag functions 
from the interval $[0,\infty)$ to the set $\mathcal{M}$ (see for details, e.g., \cite{kurtz, skorokhod}). 

\subsection{Generator of the process}
\label{ss:g-p}
We consider a continuous time $\mathcal{M}$-valued stochastic process $(\nu_t)_{t\geq 0}$ 
with an infinitesimal generator $L$ given for all bounded and 
measurable functions $\phi:\mathcal{M}\to\R$ by the formula
\begin{equation}
\label{operatorL}
\begin{aligned}
L\phi (\nu)=\int_F \int_F & \int_{F}[\phi(\nu +\delta_{z})-\phi(\nu)]m(x,y;\nu)K(x,y,dz)\nu(dx)\nu(dy)\\
&{}+\int_F [\phi(\nu -\delta_{x})-\phi(\nu)]\bigg(D(x)+I(x)\int_{F}U(x,y)\nu(dy)\bigg)\nu(dx).
\end{aligned}
\end{equation} 
The first term in right-hand side describes the mating and birth processes with dispersal of the phenotypic traits of an offspring. 
The second term stands for two kinds of deaths. The death part in above-mentioned generator 
was previously studied in \cite{fournier_meleard}. Notice that, unlike the generator appearing in \cite{fournier_meleard}, 
the operator $L$ of the form (\ref{operatorL}) is nonlinear in $\nu$ in the first term. 

To ensure the existence and uniqueness of the Markov process generated by  operator (\ref{operatorL}), we assume that $K(x,y,dz)$ is absolutely continuous with respect to the Lebesgue measure, i.e.  
\begin{equation}
\label{gestoscK}
K(x,y,dz)=k(x,y,z)\,dz,
\end{equation}
where $k$ is probability density function such that $k(x,y,\cdot)=k(y,x,\cdot)$, and there exists a constant $C>0$ and a probability density function $\overline k$ such that
\begin{equation}\label{zalK}
k\left(x,y,z-\frac{x+y}{2}\right)\leq C \overline k(z).
\end{equation} 
Moreover, we assume that there are positive constants $\underline a, \overline a, \overline D, \overline I, \overline U$ such that for every $x,y\in F$
\begin{equation}\label{upperbound}
\underline a\leq a(x,y)\leq \overline a, \quad p(x)\leq\overline p, \quad D(x)\leq\overline D, \quad I(x)\leq\overline I, \quad U(x,y)\leq\overline U.
\end{equation}
Under the above assumptions, if the initial measure $\nu_0\in\mathcal{M}$ satisfies $\E\big(\ilsk{\nu_0}{1}^q\big)<\infty$ for some number
$q\geq 1$, then $$\E\Big(\sup_{0\leq t\leq T}\ilsk{\nu_t}{1}^q\Big)<\infty$$ for any $T<\infty$, and, consequently, standard approach of Fournier and M\'el\'eard can be easily adapted to prove the existence of the Markov chain $(\nu_t)_{t\geq 0}$ with the infinitesimal generator given by formula (\ref{operatorL}) (see \cite{fournier_meleard} for detailed proofs).

\section{Macroscopic model}
\label{r_app}
\subsection{Macroscopic approximation}\label{ss:m-a}
This section contains an approximation of the process which was introduced and studied in the previous sections. 
The idea is to normalize the initial model and pass with the number of individuals to infinity, 
assuming that the "mass" of each individual becomes negligible. 
The approximation covers a case in which the rates of mating and death are unchanged. 
Only the intensity of interaction is rescaled, and tends to $0$ with the unbounded growth of the population. 
This approach leads to a deterministic nonlinear integro-differential equation 
whose solutions describe the evolution of distribution of the phenotypic trait in the population.

We consider a sequence of populations indexed by the number $N\in\N$. In the $N$-th population consisting of individuals $\mathbf x^N=\{x_1^N,\ldots, x_n^N\}$
\begin{itemize}
\item[(a)] individuals with phenotypic traits $x_i^N$ and $x_j^N$ can mate with a rate $m(x_i^N,x_j^N;\mathbf x^N)$, which is of the form (\ref{matingop1}) or (\ref{mating-as2}),
\item[(b)]  the new offspring's phenotypic trait is drawn from the distribution $K(x_i^N,x_j^N,dz)$, where $x_i^N, x_j^N$ are the phenotypic traits of the parents,
\item[(c)]  an individual with phenotypic trait $x_i^N$ can die with a rate $D(x_i^N)$,
\item[(d)]    an individual with phenotypic trait $x_i^N$ interacts with  other individuals with intensity $I(x_i^N)/N$,
\item[(e)] the competition kernel of individuals with phenotypic traits $x_i^N, x_j^N$ is a symmetric, nonnegative function $U(x_i^N,x_j^N)$.
\end{itemize}

The $N$-th population is described by a process $(\nu^N_t)_{t\geq 0}$ which is defined
in the same way as the process $(\nu_t)_{t\geq 0}$ but with the corresponding coefficients. We define a $\mathcal{M}^N$-valued Markov process $(\mu_t^N)_{t\geq 0}$ by a formula $\mu_t^N=\nu_t^N/N$. The value space for each process $(\mu^N_t)_{t\geq 0}$ is thus 
\[
\mathcal{M}^N=\bigg\{\frac{1}{N}\nu: \nu\in\mathcal{M}\bigg\}.
\]
Then the generator $L^N$ of the process $(\mu_t^N)_{t\geq 0}$ is given by
\begin{multline*}
L^N\phi (\nu)=N\int_F \int_F \int_{\R^d}\bigg(\phi \Big(\nu +\frac{1}{N}\delta_{z}\Big)-\phi(\nu)\bigg)m(x,y;\nu)K(x,y,dz)\,\nu(dx)\,\nu(dy)\\
+N\int_F \bigg ( \phi\Big(\nu -\frac{1}{N}\delta_{x}\Big)-\phi(\nu)\bigg)\bigg(D(x)+I(x)\int_{F}U(x,y)\,\nu(dy)\bigg)\,\nu(dx).
\end{multline*}
for any measurable and bounded map $\phi:\mathcal{M}^N\to\R$. 

We consider the case when the number of individuals tends to infinity with simultaneous decrease in each individual's mass to zero as $N\to\infty$.

\begin{thm}
\label{th-lp}
Assume that conditions $(\ref{gestoscK})$, $(\ref{zalK})$ and $(\ref{upperbound})$ hold, and moreover functions $a, p, k, D, I, U$ are continuous. If for some $q\geq 2$ the condition $\E\big(\ilsk{\mu_0^N}{1}^q\big)<\infty$ holds for all $N\in\N$, and the sequence $(\mu_0^N)_{N\in\N}$ converges 
weakly to a deterministic finite measure $\mu_0$ as $N\to\infty$,  
then for all $T>0$ the sequence $(\mu^N)_{N\in\N}$ converges in distribution in $\D([0,T],M(F))$ 
to a deterministic and continuous measure-valued function $\mu_t:[0,T]\to M(F)$, given by the formula
\begin{align}\label{rownanieslabe}
\ilsk{\mu_t}{f}={} &\ilsk{\mu_0}{f}+\int_0^t\int_F\int_F\int_{\R^d} f(z)m(x,y;\mu_s) k(x,y,z)\,dz\,\mu_s(dx)\,\mu_s(dy)\,ds
\notag\\
&-\int_0^t\int_F f(x)\Big( D(x)+I(x)\int_F U(x,y)\,\mu_s(dy) \Big) \mu_s(dx)\,ds,
\end{align}
for every bounded and measurable function $f\colon F\to\R.$
\end{thm}

The standard proof of the above theorem is based on \cite{kurtz} Corollary 8.16, Chapter 4, and since the mating is described by Lipschitz continuous operator mapping the space of positive, finite Borel measures with total variation norm into itself, 
it can be directly adapted, for example, from \cite{rudnicki_wieczorek}.

\subsection{Strong solutions in the space of measures}
\label{ss:densities}
Some of interesting biological examples do not satisfy either condition (\ref{gestoscK}) or (\ref{zalK}). This is the reason why in this part of the paper we study the behavior of the function $\mu_t$ given by Theorem \ref{th-lp} in more general setting. To this end, we use the following formal notation
\begin{equation}
\label{rownaniemocne}
\begin{aligned}
\frac{d}{d t} \mu_t(dz)=\int_F\int_F 
&\,m(x,y;\mu_t)\,
K(x,y,dz)\mu_t(dx)\mu_t(dy)\\
&{}-\bigg(D(z)+I(z)\int_F U(z,y)\mu_t(dy)\bigg)\mu_t(dz),
\end{aligned}
\end{equation}
of equation (\ref{rownanieslabe}) in the space of positive, finite Borel measures $M(F)$ on the set $F$ with the total variation norm 
$\norm{\nu}_{TV}=\sup\{
|\ilsk{\nu}{f}|\colon \,\,f \textrm{ -- measurable, } \sup\limits_{x\in F}|f(x)|\le 1\}$.  

\begin{thm}\label{istjednozn}
Assume that the functions $a, p, D, I, U$ and $(\ref{Kmierzalne})$ are measurable, and condition $(\ref{upperbound})$ holds. Moreover, suppose that there exist positive constants $\underline p, \underline I, \underline U$ such that 
\begin{equation}
p(x)\geq \underline p, \quad I(x)\geq \underline I, \quad U(x,y)\geq \underline U,
\end{equation} for all $x,y\in F$. If $\mu_0\in M(F)$, then there exists a unique 
solution $\mu_t$, $t\geq 0$, of equation $(\ref{rownaniemocne})$ with the initial condition $\mu_0$, and the function $t\mapsto \mu_t$ is bounded and continuous 
in the norm $\norm{\cdot}_{TV}$.
\end{thm}
\begin{proof}
Let us fix $T\geq 0$, $\delta>0$ and consider a space $C_T=C([T,T+\delta], M(F))$ 
with a norm $\norm{\mu_\cdot}_T=\sup_{t\in[T,T+\delta]} \norm{\mu_t}_{TV}$.
Define the operator $\Lambda\colon C_T\to C_T$ by the formula 
\begin{align*}
(\Lambda\mu_\cdot)(t)(dz)={}&\mu_T(dz)+\int_T^t \int_F\int_F m(x,y;\mu_s)K(x,y,dz)\,\mu_s(dx)\,\mu_s(dy)\,ds\\
&{}-\int_T^t\bigg(D(z)+I(z)\int_F U(z,y)\mu_s(dy)\bigg)\mu_s(dz)\,ds,
\end{align*}
where $\mu_T\in M(F)$ is some measure. Notice that from assumption (\ref{upperbound}) there are constants $\overline m, \hat m$ depending on $\overline p$ in the case of semi-random mating, and on $\underline a, \overline a$ in the case of assortative mating, such that for any $y\in F$ and measures $\mu, \nu\in M(F)$
\begin{equation}\label{mogranicz}
\int_F m(x,y;\nu)\,\nu(dx)\leq \overline m,
\end{equation}
and
\begin{equation}\label{mlipschitz}
\int_F\int_F|m(x,y;\mu)-m(x,y;\nu)|\,\mu(dx)\,\nu(dy)\leq \overline m \norm{\mu-\nu}_{TV}.
\end{equation}
Take functions $\mu_\cdot, \nu_\cdot\in C_T$ from the ball $B(0,2\norm{\mu_T}_{TV})$. Then
\begin{equation}\label{ist1}
\norm{\Lambda\mu_\cdot}_T \leq \norm{\mu_T}_{TV}+2\delta (\overline m+\overline D) \norm{\mu_T}_{TV}+4\delta \overline I\overline U \norm{\mu_T}_{TV}^2,
\end{equation}
and 
\begin{equation}
\label{ist2}
\norm{\Lambda\mu_\cdot-\Lambda\nu_\cdot}_T\leq \delta \big(\overline D+ 2\overline m + 4\hat m\norm{\mu_T}_{TV}^2\big)  \norm{\mu_\cdot-\nu_\cdot}_T 
+ 4\delta\overline I\overline U\norm{\mu_T}_{TV}\norm{\mu_\cdot-\nu_\cdot}_T. 
\end{equation}
Taking $\delta>0$ sufficiently small, from (\ref{ist1}) and (\ref{ist2}) it follows that  
$\Lambda$ transforms the ball $B(0,2\norm{\mu_T}_{TV})$ 
into itself,  and is Lipschitz continuous with some constant $L<1$. 
By the Banach fix point theorem there exists a unique solution in the interval $[T,T+\delta]$. 
Consequently, there exists a unique local solution of (\ref{rownaniemocne}). 

To extend a local solution to a global solution on a whole interval $[0,\infty)$
it is sufficient to show that $\delta$ can be chosen independently of $T$. 
It follows from the upper-bound of the solutions. 
To search for upper-bound of the solutions notice that
\begin{equation}
\label{ograniczenienormy}
\frac{d}{dt}\norm{\mu_t}_{TV}\leq \norm{\mu_t}_{TV} \big(\overline m -\underline I\,\underline U \norm{\mu_t}_{TV}\big).
\end{equation}
Consequently,
\[
\norm{\mu_t}_{TV}\leq \max\left\{\norm{\mu_0}_{TV}, \frac{\overline m}{\underline I\,\underline U}\right\},
\]
which completes the proof of the existence and uniqueness.

Eventually, we show that for fixed $t>0$ the measure $\mu_t$ is positive. Indeed, for any Borel set $A$ we can write 
$\frac{\partial}{\partial t} \mu_t(A)\geq -\phi(t)\mu_t(A)$, where 
$\phi(t)=\left(\overline D+\overline I\overline U\mu_t(F)\right).$
Finally, 
\[
\mu_t(A)\geq \mu_0(A) \exp\bigg\{-\int_0^t\phi(s)\,ds\bigg\} \ge 0.
\qedhere
\]
\end{proof}

The straight-forward conclusion is the following statement about the solutions in $L^1$ space.
\begin{col}
Suppose that condition $(\ref{gestoscK})$ holds. Under the assumptions of Theorem~$\ref{istjednozn}$, if $\mu_0$ has a density $u_0\in L^1$ with respect to the Lebesgue measure, then $\mu_t$ also has a density $u(t,\cdot)\in L^1$ with respect to the Lebesgue measure, which is the unique solution of the following equation
\begin{equation}
\begin{aligned}
\frac{\partial}{\partial t} u(t,z)=\int_F\int_F 
&\,m(x,y;u(t,\xi)d\xi)\,
k(x,y,z)u(t,x)u(t,y)dxdy\\
&{}-\bigg(D(z)+I(z)\int_F U(z,y)u(t,y)dy\bigg)u(t,z),
\end{aligned}
\end{equation}
with the initial condition $u(0,\cdot)=u_0(\cdot)$.
\end{col}
\begin{proof}
Take a Borel set $A$ with zero Lebesgue measure. Since the measure $K(x,y,dz)$ is absolute continuous with respect to the Lebesgue measure, $K(x,y,A)=0$ for every $x,y\in F,$ and consequently
$\frac{d}{d t} \mu_t(A)\leq -\underline D \,\mu_t(A)$, and $\mu_0(A)=0$.
Therefore $\mu_t(A)=0$ for all $t>0$, and the statement comes from the Radon-Nikodym theorem.
\end{proof}

\subsection{Boundedness, extinction and  persistence}
\label{ss:extinction}
From Theorem~\ref{istjednozn} it follows that the function $M(t)=\mu_t(F)$ is upper-bounded. Now we analyze further properties of the function $M(t)$. Let us recall that the population \textit{becomes extinct} if $\lim_{t \to\infty}M(t)=0$, and \textit{is persistent} provided $\liminf_{t \to\infty}M(t)>0$. 
\begin{pr}
\label{prlost} 
If $\inf_z D(z)\ge \sup_z p(z) $ in the case of random mating
and 
$\inf_z D(z)\ge 1$ in the case of assortative mating, then
the population becomes extinct.
If $\sup_z D(z)< \inf_z p(z) $ in the case of random mating
and $\sup_z D(z)< 1$ in the case of assortative mating, then
the population is persistent.
 \end{pr}
\begin{proof}
In the case of random mating these  properties are simple consequences of  the following  inequalities
\[
M'(t)\leq M(t)\left(\bar p - \underline D - \underline I\,\underline U M(t)\right)
\]
and
\[
M'(t)\ge M(t)\left(\underline  p - \overline D - \overline I\,\overline U M(t)\right).
\]
In the case of assortative mating we use similar inequalities replacing    $\bar p$ and $\underline  p$ by 1.
\end{proof}

\subsection{Equation on a global attractor}
\label{ss:absorbing set}
In order to describe  more precisely the  asymptotic behavior of $M(t)$,  we  need to assume that
the functions $p, D, I, U$ do not depend on $x$, and are positive.  
To avoid the extinction of the population, we additionally assume that $D<p$ in the case of random mating  and 
$D<1=:p$ in the case of assortative mating  (see Proposition \ref{prlost}).
Then,  using basic facts from the theory of differential equations,
 it is easy to see that
\begin{equation}
\label{ab:limit}
\lim_{t\to\infty} M(t)=\frac{p-D}{IU}.
\end{equation}   
The number $\frac{p-D}{IU}$ is an analogue of \textit{carrying capacity}
studied in \cite{bolker_pacala} and \cite{fournier_meleard}. 
In our case it can be viewed as a number of individuals per unit of volume 
after a long time. 

From (\ref{ab:limit}) it follows that all positive solutions converge to  the set 
\[
\mathcal A=\bigg\{\mu\in M(F)\colon \,\,\mu(F)=\frac{p-D}{IU}\bigg\}
\]
and the set $\mathcal A$ is invariant with respect to  equation $(\ref{rownaniemocne})$, i.e., 
if the initial condition $\mu_0$ belongs to  $\mathcal A$,  then 
$\mu_t \in \mathcal A$ for $t>0$. It means that  $\mathcal A$ is a \textit{global attractor} for  equation $(\ref{rownaniemocne})$ 
and it is interesting  to study  the behavior of solutions on the set  $\mathcal A$.

First, we consider  the case of random mating. Let $\mu\in \mathcal A$.
We replace $\mu_t(dx)$ by $\frac{p-D}{IU}\mu_t(dx)$ and
and $t$ by $p t$ in (\ref{rownaniemocne}), to obtain the equation
\begin{equation}
\label{rownstab}
\frac{\partial}{\partial t} \mu_t(dz)+ \mu_t(dz)=\int_F\int_F K(x,y,dz)\mu_t(dx)\mu_t(dy).
\end{equation}
The measure $\mu_t(dz)$ is  a \textit{probability measure} for all $t\ge 0$,
 i.e., $\mu_t\ge 0$ and $\mu_t(F) =1$.
In the case of assortative mating we replace $\mu_t(dz)$ by $\frac{1-D}{IU}\mu_t(dz)$ and receive
\begin{equation}\label{rownstab-am}
\frac{\partial}{\partial t} \mu_t(dz)+ \mu_t(dz)=\int_F\int_F \psi(x,y;\mu_t)K(x,y,dz)\mu_t(dx)\mu_t(dy),
\end{equation}
where
\[
\psi(x,y;\mu)=\frac{a(x,y)}{2\int_F a(x,r)\mu(dr)}+\frac{a(x,y)}{2\int_F a(y,r)\mu(dr)}.
\]

\section{Asymptotic stability in the case of random mating}
\label{s:as}
\subsection{General remarks}
\label{ss: gr}
In this section we study the convergence of solutions of equation (\ref{rownstab}) to some stationary solutions.
Equation (\ref{rownstab})  can be treated
as an evolution equation
 \begin{equation}
\label{ewol}
\mu_t'=\mathcal P\mu_t-\mu_t ,
 \end{equation}
 where the operator  $\mathcal P$ acting on the space of all probability Borel measures
on $F$ is given by the formula 
\begin{equation}
(\mathcal P\mu)(A)=\int_F\int_F K(x,y,A)\,\mu(dx)\,\mu(dy),
\end{equation}
The solution of 
(\ref{ewol}) with the initial measure $\mu_0$ 
is the deterministic process $\mu_t$ given by 
Theorem~\ref{istjednozn}. The set $\mathcal O(\mu_0):=\{\mu_t\colon t\ge 0\}$ is called the \textit{orbit} of $\mu_0$.

Since the problem of the asymptotic stability of the solutions 
of equation (\ref{ewol})  in an arbitrary $d$-dimensional space seems to be quite difficult, 
we consider only the case when $d=1$ and $F$ is a closed interval with nonempty interior. 
Generally, equation  (\ref{ewol})  has a lot of different  stationary measures  and it is rather difficult to predict the limit of a given solution. 
Assumption (\ref{srednia2}) allows us to omit this difficulty. 
Indeed, if a measure $\mu$ has a finite first moment $q$,  
then according to (\ref{srednia1}) and (\ref{srednia2}) we have
\begin{multline*}
\int_F |z|\,(\mathcal{P}\mu)(dz)=\int_F\int_F\int_F |z|K(x,y,dz)\,\mu(dx)\,\mu(dy)\\
\le \int_F\int_F (c_1+c_2|x|+c_3|y|)\, \mu(dx)\,\mu(dy)\le c_1+(c_2+c_3)\int_F|x|\,\mu(dx)<\infty
\end{multline*}
and
\[
\int_F z\,(\mathcal{P}\mu)(dz)=\int_F\int_F\int_F zK(x,y,dz)\,\mu(dx)\,\mu(dy)=\int_F\int_F \frac{x+y}{2}\, \mu(dx)\,\mu(dy)=q.
\]
Therefore,  any solution $\mu_t$ of  equation (\ref{ewol}) have the same first  moment for all $t\ge 0$. 
It means that we can restrict our  consideration only to probability Borel measures 
with the same first moment.

The following example shows why we consider solutions of  equation (\ref{ewol})
 with values in the space of probability Borel measures instead of the space of probability densities.
In this example all the stationary solutions are Dirac measures, and
any solution converges in the weak sense to some stationary measure.
\begin{example}
\label{e:2}
Let $Z$ be a random variable with values in the interval $[-1,1]$ such that $\E Z=0$ and $|Z|\not\equiv 1.$
Assume that if $x$ and $y$  
are parental  traits, then 
the phenotypic trait of an offspring
is given by 
\[
\frac{x+y}2+ Z\frac{|x-y|}{2},
\]
i.e., the trait of an offspring is distributed between the traits of parents according to the law of $Z$. 
For random variable $X$ we denote by $m_1(X)$ and $m_2(X)$ its first and second moments and by $D(X)$ its variance, i.e., $D(X)=m_2(X)-(m_1(X))^2$.
If $M_t$ is a random variable distributed by the solution $\mu_t$ of equation (\ref{ewol}) with the finite second moment, then
$\bar x:=m_1(M_t)$ is a constant, and 
\begin{equation}
\frac {d}{dt} D(M_t)= -\frac12 (1-D(Z))D(M_t).
\end{equation}
Since $D(Z)<1$, we have $\lim_{t\to\infty} D(M_t)=0$. Consequently,
$\mu_t$ convergence in the weak sense to $\delta_{\bar x}$.
\end{example}

\subsection{Wasserstein distance} 
\label{W-d}
In order to investigate asymptotic properties of the solutions, we recall some basic theory concerning Wasserstein distance between measures.
For $\alpha\ge 1$ we denote by $\mathcal{M}_{\alpha}$  the set of all probability Borel measures $\mu$ on $F$ such that 
$\int_F |z|^{\alpha}\,\mu(dz)<\infty$
 and  by $\mathcal{M}_{\alpha,q}$ the subset of $\mathcal{M}_{\alpha}$  which contains all the  measures such that
 $\int_F z\,\mu(dz)=q$.
For any two measures $\mu,\nu\in \mathcal{M}_1$, we define the \textit{Wasserstein distance} by the formula
\begin{equation}
d(\mu,\nu)=\sup_{f\in\textup{Lip}_1} \int_F f(z)(\mu-\nu)(dz),
\end{equation}
where $\textup{Lip}_1$ is a set of all continuous functions $f:F\to \R$ such that for any $x,y\in F$ 
\[
|f(x)-f(y)|\leq |x-y|.
\]

The following lemma is of a great importance in the subsequent part of the paper.

\begin{lem}
\label{waslem} 
The  Wasserstein distance between measures
 $\mu,\nu\in \mathcal{M}_{1}$ can be computed by the formula
\begin{equation}
d(\mu,\nu)=\int_F |\Phi(x)|\,dx,
\end{equation}
where $\Phi(z)=(\mu-\nu)\left(F\cap(-\infty,z]\right)$ is a cumulative distribution function of the signed measure $\mu-\nu$.
\end{lem}
\begin{proof}
Let $\Phi_{\mu}$ and 
$\Phi_{\nu}$ be the cumulative distribution functions 
of  the measures $\mu$ and $\nu$. Since these measures  have  finite first absolute moments  
we have
\[
\lim_{x\to-\infty}|x|\Phi_{\mu}(x) =\lim_{x\to -\infty}|x|\Phi_{\nu}(x)=0  
\] 
and 
\[
\lim_{x\to \infty}x(1-\Phi_{\mu}(x)) =\lim_{x\to\-\infty}x(1 -\Phi_{\nu}(x))=0.  
\] 
This gives 
\[
\limsup_{x\to \pm \infty}|f(x)|\Phi (x) \le \lim_{x\to \pm \infty}|x|\Phi (x) =0\quad\textrm{for $f\in\textup{Lip}_1$}
\] 
and if $F$ is bounded  from below or from above, then we have $\Phi(x)=0$ for  $x\notin F$. 
Since  $f$ is a locally absolutely continuous function,  integrating by parts leads to the formula
\[
d(\mu,\nu)=\sup_{f\in\textup{Lip}_1} \int_F f(z)\,d\Phi(z)=\sup_{f\in\textup{Lip}_1} -\int_F f'(z)\Phi(z)\,dz.
\]
Clearly the supremum  is taken when $f'(z)=-\sgn \Phi(z)$.
\end{proof}
Consider probability measures $\mu$ and $\mu_n$, $n\in\N$, on the set $F$. We recall that the sequence $\mu_n$ converges \textit{weakly} (or \textit{in a weak sens}) to $\mu$, if for any continuous and bounded function $f\colon F\to\R$
$$\int_F f(x)\,\mu_n(dx)\to\int_F f(x)\,\mu(dx),$$
as $n\to\infty$.
It is well-known that the convergence in Wasserstein distance implies weak convergence of measures. Moreover, the space of probability Borel measures on any complete metric space is also a complete metric space with the Wasserstein distance (see e.g. \cite{bolley, rach}). The convergence of a sequence $\mu_n$ to $\mu$ in the space $\mathcal{M}_{1,q}$ is equivalent to the following condition (see \cite{villani}, Definition 6.7 and Theorem 6.8)
\begin{equation}\tag{C}\label{warrown}
\mu_n\to\mu \mbox{ weakly, as } n\to\infty \quad \mbox{ and } \quad\lim_{R\to\infty} \limsup_{n\to\infty} \int_{F_R} |x|\mu_n(dx)=0,
\end{equation}
where $F_R:=\{x\in F\colon |x|\geq R \}.$ Fix $q\in F$, $\alpha>1$, and $m>0$.  Consider a set $\widetilde{\mathcal{M}} \subset\mathcal{M} _{1,q}$ such that
\begin{equation}
\int_F |x|^\alpha \mu(dx)\leq m
\end{equation}
for all $\mu\in\widetilde{\mathcal{M}}.$ Then the set $\widetilde{\mathcal{M}}$ is relatively compact in $\mathcal{M}_{1,q}$.
Indeed, by Markov inequality we obtain that $\mu(\{x\colon |x|\leq R\})\geq 1-m/R^\alpha$ for all $\mu\in\widetilde{\mathcal{M}}$, 
what means that the set $\widetilde{\mathcal{M}}$ is tight, and thus  $\widetilde{\mathcal{M}}$ is relatively compact in the topology of weak convergence (see e.g. \cite{billingsley}). Moreover, for $\mu\in\widetilde{\mathcal{M}}$ we have
\[
\int_{F_R} |x|\mu_k(dx)\leq \frac{1}{R^{\alpha-1}} \int_{F_R} |x|^\alpha \mu_k(dx)\leq \frac{m}{R^{\alpha-1}},
\] 
which implies the second condition in (\ref{warrown}). Consequently, the set $\widetilde{\mathcal{M}}$ is relatively compact in $\mathcal{M}_{1,q}$.

\subsection{Theorems on asymptotic stability}
\label{ss:mr}
We use the script letter $\mathcal K$ for the cumulative distribution function of the measure $K$, i.e.,
\[
\mathcal K(x,y,z)= K(x,y,F\cap (-\infty,z]).
\]
The main result of this section is the following.
\begin{thm}
\label{limitthm} 
Fix $q\in F$.
Suppose that\\ 
\noindent (i) for all $y,z\in F$ the function $\mathcal K(x,y,z)$ is absolutely continuous with respect to $x$
and for each $a,b,y\in F$ we have
\begin{equation}\label{warstab2}
\int_F \Big|\frac{\partial}{\partial x}\mathcal K(a,y,z)-\frac{\partial}{\partial x}\mathcal K(b,y,z)\Big|\,dz< 1,
\end{equation} 
\noindent (ii)  there are constants $\alpha>1$, $L<1$, and $C\ge 0$  such that for 
every $\mu\in\mathcal M_{\alpha,q}$ we have 
\begin{equation}
\label{zwartosc}
\int_F |x|^{\alpha}\mathcal{P}\mu(dx)\le C+L\int_F |x|^{\alpha} \mu(dx).
\end{equation}     
Then for every initial measure $\mu_0\in\mathcal{M}_{1,q}$ 
there exists a unique solution $\mu_t$, $t\ge 0$,  of equation $(\ref{rownstab})$ 
with values in $\mathcal{M}_{1,q}$. 
Moreover,  there exists a unique measure  
$\mu^*\in\mathcal{M}_{1,q}$
such that $\mathcal P\mu^*=\mu^*$
and for every initial measure $\mu_0\in\mathcal{M}_{1,q}$ 
the solution $\mu_t$, $t\ge 0$,  of equation $(\ref{rownstab})$ 
converges to $\mu^*$ 
in the space $\mathcal{M}_{1,q}$.
\end{thm}
The straightforward conclusion is following: 
the mean of the stable population 
$\mu^*$  is equal to the mean $q$ of the initial population.
We split the proof of Theorem~\ref{limitthm} into a sequence of lemmas. 
Denote  by $\mathcal F_q$ the set of all  cumulative distribution functions of
the signed measures of the form $\mu-\nu$, where 
$\mu,\nu\in \mathcal{M}_{1,q}$.
\begin{lem}
\label{pr:lc}
Suppose that 
for  all $y\in F$ and   $\Phi \in\mathcal F_q$, $\Phi\not\equiv  0$, 
we have 
\begin{equation}
\label{w:stab1}
2\int_F\left|\int_F \mathcal K(x,y,z)\Phi(dx)\right| \,dz <\int_F |\Phi(x)|\,dx.
\end{equation}
Then 
\begin{equation}
\label{w:stab2}
d(\mathcal{P}\mu,\mathcal{P}\nu)<d(\mu,\nu)
\end{equation}
for $\mu,\nu\in \mathcal{M}_{1,q}$, $\mu\ne \nu$.
In particular, for every initial measure $\mu_0\in\mathcal{M}_{1,q}$ 
there exists a unique solution $\mu_t$, $t\ge 0$,  of equation $(\ref{rownstab})$ 
with values in $\mathcal{M}_{1,q}$. 
\end{lem}
\begin{proof}
Since $K(x,y,\cdot)=K(y,x,\cdot)$, we can write 
\[
\mathcal{P}\mu-\mathcal{P}\nu=
2\int_F\int_F K(x,y,\cdot)\,(\mu-\nu)(dx)\,\bar \mu(dy),
\]
where $\bar\mu=(\mu+\nu)/2$.
If  $\Phi(x)$ is a cumulative distribution function of $\mu-\nu$, then
the signed measure $\mathcal{P}\mu-\mathcal{P}\nu$ 
has the cumulative distribution function of the form
\[
2\int_F\int_F \mathcal K(x,y,z)\,\Phi(dx)\,\bar \mu(dy).
\]
Hence
\begin{align*}
d(\mathcal{P}\mu,\mathcal{P}\nu)&=2\int_F \left|\int_F\int_F \mathcal  K(x,y,z)\,\Phi(dx)\,\bar \mu(dy)\right| dz\\
&\le 2\int_F \int_F\left|\int_F \mathcal K(x,y,z)\,\Phi(dx)\right|\,dz\, \bar\mu(dy)\\
&<\int_F \int_F\left|\Phi(x)\right|\,dx\,\bar\mu(dy)
=\int_F\left|\Phi(x)\right|\,dx=d(\mu,\nu)
. \qedhere
\end{align*}
\end{proof}

\begin{lem}
\label{pr:lc2}
Suppose that condition (i) of Theorem~$\ref{limitthm}$ is fulfilled. Then  condition $(\ref{w:stab1})$ holds.
 \end{lem}
\begin{proof}
Take a $\Phi \in\mathcal F_q$
and denote $\Phi^+(x)=\max\{0,\Phi(x)\}$ and $\Phi^-(x)=\max\{0,-\Phi(x)\}$. 
Since $\Phi$ is the  cumulative distribution functions of a signed measure $\mu-\nu$, 
where $\mu,\nu\in \mathcal{M}_{1,q}$  we have
\[
\int_F\Phi(x)\,dx=-\int_F x\,\Phi(dx)=\int_F x\,\Phi_{\nu}(dx)-\int_F x\,\Phi_{\mu}(dx) =0,
\]
and, consequently, 
\begin{equation}
\label{p-m}
\int_F\Phi^+(x)\,dx=\int_F\Phi^-(x)\,dx=\frac12\int_F|\Phi(x)|\,dx.
\end{equation}
Since   $\Phi^+$ and $\Phi^-$ are nonnegative functions and have the same integral, condition
(\ref{warstab2}) implies
\begin{equation}
\label{warstab3}
\int_F \Big|\int_F\frac{\partial}{\partial x}\mathcal K(x,y,z)\Phi^+(x)\,dx
-\int_F\frac{\partial}{\partial x}\mathcal K(x,y,z)\Phi^-(x)\,dx\Big|\,dz< \int_F \Phi^+(x)\,dx.
\end{equation} 
Integrating $\int_F \mathcal K(x,y,z)\Phi(dx)$ by parts we obtain
\begin{equation}
\label{warstab4}
\int_F \mathcal K(x,y,z)\Phi(dx)=-\int_F \frac{\partial}{\partial x}\mathcal K(x,y,z)\Phi(x)\,dx
.
\end{equation}
From  (\ref{warstab3}) and (\ref{warstab4}) it follows 
\begin{equation*}
2\int_F\left|\int_F \mathcal  K(x,y,z)\Phi(dx)\right| \,dz <2\int_F \Phi^+(x)\,dx=\int_F |\Phi(x)|\,dx.
\qedhere
\end{equation*}
\end{proof}

\begin{lem}
\label{pr:lc3}
Assume that $d(\mathcal{P}\mu,\mathcal{P}\nu)<d(\mu,\nu)$
for all $\mu,\nu\in \mathcal{M}_{1,q}$, $\mu\ne \nu$.
Let $\mu_0,\nu_0\in \mathcal{M}_{1,q}$ and denote by 
$\mu_t$ and $\nu_t$, respectively, the solutions of 
equation $(\ref{ewol})$ in the space of 
probability Borel measures
on $F$. Then $\mu_t,\nu_t\in \mathcal{M}_{1,q}$ for $t\ge 0$
and $d(\mu_t,\nu_t)<d(\mu_r,\nu_r)$ for $0\le r<t\le T$ provided that
$\mu_T\ne \nu_T$. 
 \end{lem}
\begin{proof}
Since $\mathcal{P}(\mathcal{M}_{1,q})\subset \mathcal{M}_{1,q}$
every solution of (\ref{ewol}) with the initial value from the set $\in\mathcal{M}_{1,q}$ remains in this
set for all $t\ge 0$.  Any solution $\mu_t$ of (\ref{ewol}) satisfies the following integral equation
\begin{equation}
\label{stabl1}
\mu_t=e^{r-t}\mu_r+\int_r^t e^{s-t}\mathcal{P}\mu_s\,ds.
\end{equation}     
Let $\mu_t$ and $\nu_t$ be solutions of  (\ref{ewol}) with values in
$\mathcal{M}_{1,q}$ and such that $\mu_T\ne \nu_T$.
Then $\mu_t\ne \nu_t$ for $t\le T$ and from 
(\ref{stabl1}) it follows that
\begin{align*}
d(\mu_t,\nu_t)&\le e^{r-t}d(\mu_r,\nu_r)+\int_r^t e^{s-t}d(\mathcal{P}\mu_s,\mathcal{P}\nu_s)\,ds\\
&<e^{r-t}d(\mu_r,\nu_r)+\int_r^t e^{s-t}d(\mu_s,\nu_s)\,ds
\end{align*}
for  $0\le r<t\le T$. Let $\alpha(s)=e^sd(\mu_s,\nu_s)$. Then
\[
\alpha(t)<\alpha(r)+\int_r^t\alpha(s)\,ds
\]
and from 
Gronwall's lemma it follows that 
$\alpha(t)<\alpha(r)e^{t-r}$, which gives
$d(\mu_t,\nu_t)<d(\mu_r,\nu_r)$.  
\end{proof}

\begin{lem}
\label{l:rc}
Assume that condition (ii) of Theorem~$\ref{limitthm}$ is fulfilled.
Then for every initial measure   $\mu_0\in\mathcal M_{\alpha,q}$ its orbit $\mathcal O(\mu_0)$ is a relatively compact subset of  
$\mathcal M_{1,q}$. Moreover,  ${\rm cl\,}\mathcal O(\mu_0)\subset \mathcal M_{\alpha,q}$, where 
${\rm cl\,}\mathcal O(\mu_0)$ denotes the closure of $\mathcal O(\mu_0)$ in $(\mathcal M_{1,q},d)$.  
\end{lem}
\begin{proof} 
We take a $\mu_0\in\mathcal M_{\alpha,q}$ and define $\beta(t)=\int_F |x|^{\alpha} \mu_t(dx)$.
From (\ref{zwartosc}) and (\ref{stabl1}) with $r=0$ it follows
\begin{equation}
\label{zwart2}
\beta(t)\le e^{-t}\beta(0)+\int_0^t e^{s-t}(C+L\beta(s))\,ds.
\end{equation}     
We define  $\gamma(t)=e^t\big(\beta(t)-\frac{C}{1-L}\big)$. Then
(\ref{zwart2}) implies that
\begin{equation}
\label{zwart1}
\gamma(t)\le \gamma(0)+L\int_0^t \gamma(s)\,ds,
\end{equation}     
and again according to Gronwall's lemma  $\gamma(t)\le \gamma(0)e^{Lt}$,
which gives
\[
\beta(t)\le  \frac{C}{1-L}+ \Big(\beta(0)-\frac{C}{1-L}\Big)e^{(L-1)t}\le 
\max
\Big\{\beta(0),\frac{C}{1-L}\Big\}.
\]
Thus, there exists $m>0$  depending on $\mu_0$ and $\alpha>1$ 
such that  $\int_F |x|^{\alpha} \mu_t(dx)\le m$ for $t\ge 0$. Consequently the orbit is a relatively compact subset of $\mathcal{M}_{1,q}$, and moreover ${\rm cl\,}\mathcal O(\mu_0)\subset \mathcal M_{\alpha,q}$.
\end{proof}

Let $\{S(t)\}_{t\ge 0}$ be a family of transformations of $\mathcal M_{1,q}$ defined by
$S(t)\mu_0=\mu_t$, where
$\mu_t$ is the solution of  $(\ref{ewol})$ with the initial condition $\mu_0$.
For $\mu\in \mathcal M_{1,q}$ we define 
the $\omega$-\textit{limit set} by 
\[
\omega(\mu)=\big\{\nu\colon \nu=\lim_{n\to\infty}\mu_{t_n}\textrm{ for a sequence $(t_n)_{n\in\mathbb N}$ with 
$\lim_{n\to\infty}t_n=\infty$}\big\}.   
\]
\begin{proof}[Proof of Theorem~$\ref{limitthm}$]       
Take a measure $\mu\in\mathcal M_{\alpha,q}$. According to Lemma~\ref{l:rc} 
the orbit of $\mu$ is a relatively compact subset  of  
$\mathcal M_{1,q}$. From this it follows that $\omega(\mu)$ is a nonempty compact set and for $t>0$ we have  
$S(t)(\omega(\mu))=\omega(\mu)$.
First we check that  $\omega(\mu)$ is a singleton.
Indeed, if  $\omega(\mu)$ has more than one element, then since   
$\omega(\mu)$ is a compact set, we can find two elements $\nu_1$ and $\nu_2$ 
in $\omega(\mu)$ with the maximal distance $d(\nu_1,\nu_2)$.
But since $S(t)(\omega(\mu))=\omega(\mu)$, then for given $t>0$ there exist     
$\bar\nu_1$ and $\bar \nu_2$ in  $\omega(\mu)$ such that $S(t)\bar\nu_1=\nu_1$  
and $S(t)\bar\nu_2=\nu_2$. Now from condition (i) and Lemmas~\ref{pr:lc}, \ref{pr:lc2}, and \ref{pr:lc3}
 it follows that  
\[
d(\nu_1,\nu_2)=
d(S(t)\bar\nu_1,S(t)\bar\nu_2)<
d(\bar\nu_1,\bar\nu_2),
\]
which contradicts the definition of $\nu_1$ and $\nu_2$.
Let  $\omega(\mu)=\{\mu^*\}$. 
Then $S(t)\mu^*=\mu^*$ for $t\ge 0$ and, consequently, $\mathcal P\mu^*=\mu^*$.
Since the orbit $\mathcal O(\mu)$ 
is relatively compact, we have $\lim_{t\to\infty} S(t)\mu=\mu^*$. 
According to Lemmas~\ref{pr:lc} and \ref{pr:lc2} the operator $\mathcal P$ has only one fixed point
what means that the limit $\lim_{t\to\infty} S(t)\mu$ does not depend on 
$\mu\in\mathcal M_{\alpha,q}$.   
Now, we consider a measure $\mu \in\mathcal M_{1,q}$.
The set $\mathcal M_{\alpha,q}$ is dense in the space $\mathcal M_{1,q}$.
Thus, for every $\varepsilon>0$ we can find $\bar \mu\in \mathcal M_{\alpha,q}$ such that
$d(\mu,\bar\mu)< \varepsilon$. Moreover, since 
$\lim_{t\to\infty} S(t)\bar\mu=\mu^*$ we find $t_{\varepsilon}$ such that
 $d(S(t)\bar\mu,\mu^*)< \varepsilon$  for $t\ge  t_{\varepsilon}$.
 Since the operators $S(t)$ are contractions we have 
\[
d(S(t)\mu,\mu^*)\le d(S(t)\mu,S(t)\bar\mu)+d(S(t)\bar\mu,\mu^*)<2\varepsilon 
 \]  
 for  $t\ge  t_{\varepsilon}$, which completes the proof.
\end{proof}


We can strengthen the thesis of Theorem~\ref{limitthm}, if we additionally assume that for all $x,y\in F$ the measure $K(x,y,dz)$ has the density $k(x,y,z)$ and $k$ is a bounded and continuous function.
\begin{thm}
\label{limitthm2} 
Assume that $k$ is a bounded and continuous function, and $k$ satisfies  assumptions of Theorem~{\rm{\ref{limitthm}}}.
Then the stationary measure $\mu^*$ is absolutely continuous with respect to the Lebesgue measure 
 and has a continuous  and bounded density $u_*(x)$. Moreover,
for every $\mu_0\in\mathcal{M}_{1,q}$ the solution $\mu_t$ of equation $(\ref{rownstab})$  can be written in the form
$\mu_t=e^{-t}\mu_0+\nu_t$, where $\nu_t$ are  absolutely continuous measures, have continuous  and bounded densities  $v_t(x)$,
which converge uniformly to  $u_*(x)$.
\end{thm}
\begin{proof}
Since $k$ is a continuous and bounded function and $\mu^*$ is a probability measure, 
\[
u_*(z):=\int_{F}\int_{F}k(x,y,z)\, \mu^*(dx)\,\mu^*(dy)
\]
is a continuous bounded function and $u_*$ is a density of $\mu^*$ because  $\mu^*$ is a fixed point of  the operator $\mathcal{P}$.
For any initial measure $\mu_0\in\mathcal{M}_{1,q}$  the solution $\mu_t$ of  $(\ref{rownstab})$ 
satisfies the equation
\[
\mu_t=e^{-t}\mu_0+\int_0^t e^{s-t}\mathcal{P}\mu_s\,  ds.
\]
For each $s\ge 0$ the measure $\mathcal{P}\mu_s$ has a continuous and bounded density $\bar u_s(x)$.
Since the function $\varphi\colon [0,\infty)\to \mathcal{M}_{1,q}$ given by 
$\varphi(s)=\mu_s$ is continuous and $\lim_{s\to \infty}\mu_s=\mu^*$,
 the function $\psi\colon [0,\infty)\to C_b(F)$
given by 
$\psi (s)=\bar u_s$ is continuous and 
$\lim_{s\to \infty}\bar u_s=u_*$.
Thus   the measures $\nu_t=\int_0^t e^{s-t}\mathcal{P}\mu_s\,  ds$ 
have continuous  and bounded densities  $v_t(x)$ and $v_t$ converges uniformly to $u_*$
as $t\to\infty$.
\end{proof}

\subsection{Examples}
\label{ss:ex}
Now, we study two biologically reasonable forms of $K$, which satisfy  conditions (i) and (ii) of Theorem~\ref{limitthm}.
\begin{example} 
\label{e:1}
We suppose that if $x$ and $y$ are parential traits, then the phenotypic trait of the offspring is of the form 
\[
\frac{x+y}{2}+Z,
\]
where $Z$ is a $0$-mean random variable, $\E Z^2<\infty$  and $Z$ has a positive density $h$. 
Then
\begin{equation}
k(x,y,z)=h\left(z-\tfrac{x+y}{2}\right)
\end{equation}
and
\[
\frac{\partial}{\partial x}\mathcal K(x,y,z)=-\frac12 h\left(z-\tfrac{x+y}{2}\right).
\]
The condition (i) is equivalent to the inequality  
 \[
 \int_{-\infty}^{\infty}|h(z-\alpha)-h(z-\beta)|\,dz<2 
\]
for all $\alpha,\beta\in\mathbb R$,  which is a simple consequence of the assumption that $h$ is a
positive density. Now we check that condition (ii) holds with $\alpha=2$. We have
\begin{align*}
\int_{-\infty}^{\infty} z^2(\mathcal{P}\mu)(dz)&\le
\int_{-\infty}^{\infty}\int_{-\infty}^{\infty}\int_{-\infty}^{\infty} 
 \Big((z-\tfrac{x+y}{2})^2 +(z-\tfrac{x+y}{2})(x+y) +\tfrac{(x+y)^2}{4}\Big)  
\\
&\hskip2cm \times h\left(z-\tfrac{x+y}{2}\right)dz\,\mu(dx)\mu(dy)\\
&\le \E Z^2+\int_{-\infty}^{\infty}\int_{-\infty}^{\infty} \tfrac{(x+y)^2}{4}\mu(dx)\mu(dy)\\
&\le \E Z^2+\frac12 q^2+\frac12 \int_{-\infty}^{\infty} x^2\,\mu(dx).
\end{align*}

If we additionally assume that the density $h$ is a continuous function, then 
according to Theorem~\ref{limitthm2} 
the limit measure $\mu^*$ has a continuous and bounded density $u_*$, 
$\mu_t=e^{-t}\mu_0+v_t(x)\,dx$, and $v_t$ converges uniformly to 
$u_*$.
  
Now we determine 
the limiting distribution $\mu^*$. 
Densities  of  the measures $\mu^*$ and $\mu_0$ have the same first moment $q$ and
$u_*$ satisfies the equation
\begin{equation}
\label{eq:gi} 
u_*(z)=\int_\R\int_\R  h\Big(z-\frac{x+y}2\Big)\,u_*(x)\,u_*(y)\,dx\,dy. 
\end{equation}
Observe that if a  probability density $f$ satisfies  (\ref{eq:gi}) and $\int_\R xf(x)\,dx =0$,  then $f(x-q)$  also satisfies  
 (\ref{eq:gi}) and has the first moment $q$. Since  $u_*$ is a unique solution of
(\ref{eq:gi}) with the first moment $q$ we have $u_*(x)=f(x-q)$ for $x\in \R$.
Now we construct the density $f$. 
Consider an infinite sequence of i.i.d. random variables 
\[
Z_{01},Z_{11},Z_{12},Z_{21},\dots,Z_{24},Z_{31},\dots,Z_{38},\dots
\] 
with  the density $h$
and define the random variable 
\[
Y=Z_0+\frac{Z_{11}+Z_{12}}{2}+\frac{Z_{21}+\dots+Z_{24}}{4}+\frac{Z_{31}+\dots+Z_{38}}{8}+\dots \,\,.
\]
Then $\E Y=0$.  If $Y_1$ and $Y_2$ are  two independent copies of  $Y$ and 
$Z$ is a random variable with density $h$ independent of $Y_1$ and $Y_2$,
then 
\begin{equation}
\label{irw}
Y\overset d =Z+\frac{Y_1+Y_2}2.
\end{equation}
It means that the density $f$ of $Y$ satisfies (\ref{eq:gi}).
Let $h_n(x)=2^nh(2^nx)$ for $n\ge 0$ and $x\in \R$. From the definition of the random variable $Y$ it follows that
\[
f=h_0*h_1^{*2}*h_2^{*4}*h_3^{*8}*\dots,
\]
where  $f*g$ denotes the  convolution of the functions $f$ and $g$. 

From (\ref{irw}) it follows immediately that $\E Y^2=2\E Z^2$. For instance, if $Z$ has the normal distribution 
with the zero expectation and the  standard deviation $\sigma$, then 
$Y$ has also  the normal  distribution 
with the zero expectation and the  standard deviation $\sqrt{2}\sigma$.
\end{example}

\begin{example}
\label{e:3}
As in Example~\ref{e:1} we 
 suppose that if $x$ and $y$ are parential traits, then the phenotypic trait of the offspring is of the form 
\[
(x+y)Z,
\]
where $Z$ is a random variable 
with values in the interval $[0,1]$, 
and has a density $h$ such that 
\begin{equation}\label{tjonwuzal}
\int_0^\infty xh(x)\,dx=\frac{1}{2}.
\end{equation}
Then  $F=[0,\infty)$ and the function $k$ is the form
\begin{equation}
\label{tjonwu}
k(x,y,z)=\frac{1}{x+y}h\left(\frac{z}{x+y}\right)
\end{equation}
for $z\in [0,x+y]$ and $k(x,y,z)=0$ otherwise. 
Equation (\ref{ewol}) with the kernel $k$ given by
 (\ref{tjonwu}) is known as the \textit{general version of Tjon--Wu equation}.
If  $h=\1_{[0,1]}$, then this equation is  the Tjon--Wu version 
of the Boltzmann equation 
(see \cite{boby, krook, tjon}). 
The asymptotic stability  of the classical  Tjon--Wu equation
in $L^1$ 
space was proven by Kie{\l}ek (see \cite{kielek}).
Lasota and Traple (see \cite{lasota, lastrap}) proved stability  
in the general case  but in the sense of the weak convergence of measures.
If we assume additionally that the support of $h$ contains an interval $(0,\varepsilon)$, $\varepsilon>0$, then 
this result follows immediately from Theorem~\ref{limitthm}.   
Indeed, in that case one can easily compute
\[
\frac{\partial}{\partial x}\mathcal K(x,y,z)=-h\left(\frac{z}{x+y}\right)\frac{z}{(x+y)^2}.
\]
Now, condition (\ref{warstab2}) 
is equivalent to the inequality  
 \begin{equation}
 \label{e:ineq1}
 \int_{0}^{\infty}\big|h(\tfrac {z}{\alpha})\tfrac {z}{\alpha^2}-h(\tfrac {z}{\beta})\tfrac{z}{\beta^2}\big|\,dz<1 
\end{equation}
for all $\alpha,\beta>0$.
This inequality is a simple consequence of positivity of $h$ on the interval $(0,\varepsilon)$
and of the following condition   
\[
 \int_{0}^{\infty}h(\tfrac {z}{\alpha})\tfrac {z}{\alpha^2}\,dz=\int_{0}^1h(x)x\,dx=\tfrac12. 
\]
Now we check that the condition (ii) holds with $\alpha=2$. We have
\begin{align*}
\int_{-\infty}^{\infty} z^2(\mathcal{P}\mu)(dz)&\le
\int_{-\infty}^{\infty}\int_{-\infty}^{\infty}\int_{-\infty}^{\infty}
\frac{z^2}{x+y}h\left(\frac{z}{x+y}\right) dz\,\mu(dx)\mu(dy)\\
&\le 
\E Z^2 \int_{-\infty}^{\infty}\int_{-\infty}^{\infty}
(x+y)^2 \,\mu(dx)\mu(dy)\\
&\le 
\E Z^2\Big (2q^2+2 
 \int_{-\infty}^{\infty}
y^2 \,\mu(dy)\Big)\le  2q^2\E Z^2+L \int_{-\infty}^{\infty} y^2 \,\mu(dy),
\end{align*}
where $L=2\E Z^2$.
Since $0\le Z\le 1$, we have $L=2\E Z^2<2\E Z=1$. 
\end{example}

\begin{rem}
The kernel $k$ in Example~\ref{e:3} is not a continuous function even if  
 the density $h$ is a continuous and we cannot  apply directly
Theorem~\ref{limitthm2} in this case. But it not difficult to check that if $q>0$ then
$\mu^*(\{0\})=0$ and  to prove that the invariant measure $\mu^*$ has a density $u_*$
and $u_*$ is a continuous function on the interval $(0,\infty)$.
Moreover, repeating the proof of  Theorem~\ref{limitthm2} one can check that
$\mu_t=e^{-t}\mu_0+v_t(x)\,dx$, and $v_t$ converges uniformly to 
$u_*$ on the sets $[\varepsilon,\infty)$, $\varepsilon>0$. 
In particular, if we consider equation (\ref{ewol}) on the space of probability densities, then
every solution converges to
$u_*$ in $L^1[0,\infty)$.     
\end{rem}

\section{Conclusion}
\label{s:conclusion}
In the paper we presented some phenotype structured population  models with the sexual reproduction.
We consider both random and assortative mating. Our starting point is an individual-based model 
which clearly explains all interactions between individuals. A limit passage with the number of
individual to infinity leads to a macroscopic model which is a nonlinear evolution equation. We give some conditions
which guarantee the global existence of solutions, persistence of the population
and convergence of its size to some stable level. Next, we consider only a population with random mating      
and under suitable assumptions we prove that the phenotypic profile of the population converges to a stationary profile.

It would be interesting to study analytically long-time behavior of the phenotypic  profile of population with assortative mating.  
Some numerical results presented in the paper \cite{DBLD} suggest that also in this case one can expect convergence of the phenotypic  profile to multimodal  limit distributions. 
This result suggests that assortative mating can lead to polymorphic population and adaptive speciation.
We hope that our methods invented to asymptotic analysis of populations with random mating 
will be also useful in the case of assortative mating.
In order to do it, we probably need to modify the model of assortative mating (\ref{mating-asd2})  presented in Section~\ref{r_def}, because it has a disadvantage that the mating rate does not satisfy  the condition  $\sum_{j=1}^n m(x_i,x_j)= 1$ for all $i$.
We can construct a new model 
which corresponds to the same preference function $a(x,y)$ with  
a symmetric mating rate $m$ which  has  the above property.
In order to do this we look for constants $c_1,\dots,c_n$ depending on the state of population  such that 
\begin{equation}
\label{mating-as3}
m(x_i,x_j;\mathbf x)=(c_i+c_j)a(x_i,x_j)
\end{equation}
and $\sum_{j=1}^n m(x_i,x_j;\mathbf x)= 1$ for all $i$. In this way we obtain a system of linear equations for $c_1,\dots,c_n$:
\begin{equation}
\label{mating-mc}
\sum_{j=1}^n b_{ij}c_j=1,\quad \textrm{for $i=1,\dots,n$},
\end{equation}
where
$b_{ij}=a(x_i,x_j)$ for $i\ne j$ and $b_{ii}=a(x_i,x_i)+\sum_{l=1}^n a(x_i,x_l)$.
Since the matrix $[b_{ij}]$ has positive entries and the dominated main diagonal, system (\ref{mating-mc}) has a unique and positive solution.
Passage with the number of individuals to infinity leads 
to the following  mating rate
 \begin{equation}
\label{mating-asd3}
m(x,y;\mu)=(c(x;\mu)+c(y;\mu))a(x,y),
\end{equation}
where the function $c(x;\mu)$  depends on phenotypic distribution 
$\mu$, and satisfies the following Fredholm equation of the second kind
\begin{equation}
\label{Volt}
c(x;\mu)\int_F a(x,y)\,\mu(dy)+
\int_F c(y;\mu)a(x,y)\,\mu(dy)=1.
\end{equation}

One can introduce a general model which covers both a semi-random and assortative mating.   
Let $p(x)$ be the initial capability of mating
of an individual with the phenotypic trait $x$ and 
$a(x,y)$ be a symmetric nonnegative preference function. 
Now we can define a \textit{cumulative preference function} by 
$\bar a(x,y)=a(x,y)p(x)p(y)$.  The 
mating rate $m$ is a symmetric function given by (\ref{mating-as3})
with $a$ replaced  by  $\bar a$
and we assume that 
$\sum_{j=1}^n m(x_i,x_j;\mathbf x)=p(x_i)$ for each $i,j$. 
The mating rate in the infinitesimal model  is of the form
 \begin{equation}
\label{mating-asd3n}
m(x,y;\mu)=(c(x;\mu)+c(y;\mu))a(x,y)p(x)p(y),
\end{equation}
where the function $c(x;\mu)$ satisfies the following  equation
\begin{equation}
\label{Volt4}
c(x;\mu)\int_Fa(x,y)p(y)\,\mu(dy)+
\int_F c(y;\mu)a(x,y)p(y)\,\mu(dy)=1.
\end{equation}
In particular, in the semi-random case we have $a\equiv 1$
and $c\equiv1/\int_F 2p(y) u(y)\,dy$
and the mating rate is given by  (\ref{mating-rd}).
Let us recall that in the general case $c$  is not only a function of  $x$ but 
it is also depends on $\mu$ and therefore,  the proofs 
of results from Sections~\ref{r_mod} and \ref{r_app}  cannot be automatically adopted to these 
models.

\end{document}